\documentclass[10pt]{article}
\usepackage{amsmath}
\usepackage{amssymb}
\usepackage{stackrel,amssymb}
\usepackage{amsfonts}
\usepackage{amsthm}
\usepackage{slashed}
\usepackage{tikz-cd}
\usepackage{mathrsfs}
\usepackage[all]{xy}
\usepackage{mathtools}

\usepackage{mathabx,epsfig}

\usepackage{amssymb,amsfonts}
\usepackage{enumerate}
\usepackage{enumitem}
\usepackage{mathrsfs}
\usepackage{stmaryrd}
\usepackage{hyperref}
\usepackage{cleveref}

\newcommand{\F}{\mathbb{F}}

\usepackage{relsize}
\usepackage[bbgreekl]{mathbbol}
\usepackage{amsfonts}
\DeclareSymbolFontAlphabet{\mathbb}{AMSb} 
\DeclareSymbolFontAlphabet{\mathbbl}{bbold}

\newtheorem{thm}{Theorem}[subsection]

\newtheorem{lem}[thm]{Lemma}

\theoremstyle{definition}
\newtheorem{defn}[thm]{Definition}

\theoremstyle{remark}
\newtheorem{rem}[thm]{Remark}

\begin{document}

\newcommand\restr[3]{{
  \left.\kern-\nulldelimiterspace 
  #1 
  \vphantom{\big|} 
  \right|_{#2} 
  }}

\makeatletter
\renewcommand{\@seccntformat}[2]{%
  \ifcsname prefix@#1\endcsname
    \csname prefix@#1\endcsname
  \else
    \csname the#1\endcsname\quad
  \fi}
\makeatother

\makeatletter
\newcommand{\colim@}[3]{%
  \vtop{\m@th\ialign{##\cr
    \hfil$#1\operator@font colim$\hfil\cr
    \noalign{\nointerlineskip\kern1.5\ex@}#2\cr
    \noalign{\nointerlineskip\kern-\ex@}\cr}}%
}
\newcommand{\colim}{%
  \mathop{\mathpalette\colim@{\rightarrowfill@\textstyle}}\nmlimits@
}
\makeatother

\newcommand\rightthreearrow{%
        \mathrel{\vcenter{\mathsurround0pt
                \ialign{##\crcr
                        \noalign{\nointerlineskip}$\rightarrow$\crcr
                        \noalign{\nointerlineskip}$\rightarrow$\crcr
                        \noalign{\nointerlineskip}$\rightarrow$\crcr
                }%
        }}%
}

\title{A geometric model for mod $p$ bordism}         
\author{Kiran Luecke}        
\date{\today}          
\maketitle
$$\textit{Dedicated to Hartmut Luecke on the occasion of his 60th birthday.}$$

\begin{abstract}
In this note I give a positive solution to Bullett's conjecture (posed in \cite{bullet}) regarding a geometric presentation of the universal mod $p$ oriented ring spectrum.
\end{abstract}

\tableofcontents

\section{Introduction}

Chromatic homotopy theory begins with the observation that there is a correspondence between complex oriented cohomology theories and formal group laws. In  \cite{quillen} Quillen proves that $MU$, the cohomology theory of complex bordism, corresponds to the \textit{universal} formal group law. In that same paper Quillen also proves that $MO$, the cohomology theory of unoriented bordism, corresponds to the \textit{universal} $\mathbb{F}_2$-formal group law. In \cite{bullet} Bullett gives a natural definition $\mathbb{F}_p$-formal group laws for larger primes (cf. Def \ref{zpfgldef}), and exhibits the existence of a cohomology theory $V$ which corresponds to the universal $\mathbb{F}_p$-formal group law. Note that $MU$ and $MO$ are geometrically defined cohomology theories, which is of interest since it links algebraic topology with manifold geometry. While Bullett's construction of $V$ certainly features a lot of manifold geometry, in the end the cohomology theory $V$ is presented as an algebraically specified sub-theory of a different, geometrically defined cobordism theory. However Bullett does conjecture\footnote{Technically, he writes that `...one might conjecture..." but I will take a charitable reading of that wording, since he turns out to be right!} a certain geometric description of $V$ in \cite{bullet}. In this paper I will give a positive solution to that conjecture.

.

\ 

\textit{Notation and Disclaimers}: The symbol $p$ denotes an odd prime. All discrete rings are implicitly graded-commutative, so that for example if $a$ is an odd degree element then the free $\F_p$-algebra $\F_p[a]$ is really an exterior algebra, i.e. $a^2=0.$ $C_p$ denotes the cyclic group of order $p$. $C_p^n\rtimes\Sigma_n$ denotes the semi-direct product defined by the permutation action of $\Sigma_n$ on $C_p^n$. To streamline notation I will use $BC_p$ to denote both the classifying space of $C_p$ and its suspension spectrum $\Sigma^\infty_+ BC_p$ because it is almost always the latter that is being referred to in this note. An $MU$-structure is a stably almost complex structure. An $MU$-manifold is a manifold with an $MU$-structure.

\section{The set up and the technical results}

\subsection{$\mathbb{F}_p$-formal group laws and Bullett's mod-$p$ bordism spectra}
\subsubsection{$\mathbb{F}_p$-formal group laws}
\begin{defn} Let $E$ be a complex oriented homotopy ring spectrum. A mod $p$ orientation of $E$ consists of the data of 
\begin{enumerate}
    \item a complex orientation of $E$ with universal chern class $c_E\in E^2BU(1)$ (i.e. a homotopy ring map $MU\rightarrow E$)
    \item a class $e\in E^1BC_p$ such that if $i:BC_p\rightarrow BU(1)$ denotes the inclusion and $x=i^*c_E$ then $E^*BC_p\simeq E^*[[x,e]]/(e^2)$ and $e$ restricts to the generator of $E^1S^1$.
\end{enumerate}
\end{defn}

\begin{rem} Some people prefer to define a universal chern class $c_E$ as a choice of factorization of the unit map $\mathbb{S}\rightarrow E$ through the bottom cell $\mathbb{S}\rightarrow\Sigma^{-2}BU(1)$. Similarly the class $e$ in item 2 is equivalent to a factorization of the unit map through the bottom cell $\mathbb{S}\rightarrow \Sigma^{-1}BC_p$.
\end{rem}
\begin{rem}
Let $E$ be a mod $p$ oriented homotopy ring spectrum with formal group lay $F(x,y)$. Then the $p$ series $[p]_F(x)$ is zero, because the pullback $i^*:E^*BU(1)\rightarrow E^*BC_p$ is both injective and kills the $p$ series.
\end{rem}

Just as a formal group law is the structure displayed by a complex oriented cohomology theory at $BU(1)$, an $\mathbb{F}_p$-formal group law will be the structure displayed by a mod $p$ oriented cohomology theory at $BC_p$. The presence of the second generator and the oddity of its degree make things a little tedious to write out.

\begin{defn}\label{zpfgldef} An $\mathbb{F}_p$-formal group law over a graded commutative $\mathbb{F}_p$-algebra $R$ is a pair of power series $F_1,F_2\in R[[x_1,x_2,e_1,e_2]]$ $|x_i|=2$, $|e_i|=1$ such that the following hold: write $\xi_i=(x_i,e_i)$ and $F(\xi_1,\xi_2)=(F_1(\xi_1,\xi_2),F_2(\xi_1,\xi_2))$. Then
\begin{enumerate}
    \item $F(0,\xi)=F(\xi,0)=0$ (identity)
    \item $F(\xi_1,F(\xi_2,\xi_3))=F(F(\xi_1,\xi_2),\xi_3)$ (associativity)
    \item $F(\xi_1,\xi_2)=F(\xi_2,\xi_1)$ (commutativity)
    \item the $p$-fold iterate $F(\xi,F(\xi,...,\xi)...)$ is zero ($p$-series is zero)
    \item $F_2$ is independent of $e_1$ and $e_2$ (ordinary formal group law)
\end{enumerate}

\end{defn}

\begin{lem}
Let $E$ be a mod $p$ oriented cohomology theory. Then the multiplication map $m:BC_p\times BC_p\rightarrow BC_p$ defines an $\mathbb{F}_p$ formal group law over $E^*$.
\end{lem}
\begin{proof}
The power series $F_1$ and $F_2$ are the images under $m^*$ of the classes $e$ and $x$. The five conditions follow from the properties of $m$.
\end{proof}

As in the case of ordinary formal group laws, it is clear that there is a ring carrying a universal formal group law---it is the quotient of a big old free graded-commutative algebra on generators representing the coefficients of $F_1$ and $F_2$ modulo the relations imposed by the five conditions above. The difficulty is calculating what that ring really looks like, just like in the case of ordinary formal group laws and Lazard's theorem.

\begin{defn}\label{modplaz}
Define the mod $p$ Lazard ring $L_p$ to be the ring carrying the universal $\mathbb{F}_p$-formal group law. 
\end{defn}
\begin{thm} (Bullet \cite{bullet} 1.10) 
$$L_p\simeq\mathbb{F}_p[a_p,b_r,s_r]\ \ |a_p|=2p,\ |b_r|=2r,\ |s_r|=2r+1,\ p,r>0,\ r\neq p^k-1.$$
\end{thm}

\subsubsection{mod-$p$ bordism spectra}

In this section I recall Bullett's construction of mod $p$ bordism spectra. See \cite{bullet} and Chapter 4 of [3] for details.

\begin{defn}\label{corners} (cf. \cite{bullet} Def. 1.4) A manifold-with-(unlabelled)-corners of dimension $k$ is a ``manifold" modeled on the spaces $U_j=\{x\in\mathbb{R}^k | x_1,...,x_j\geq 0\}$. The strata of each of the model spaces induce strata in $M$. The codimension $j$ stratum is called the $j$-corners of $M$. If the only nonempty strata are of codimension 0 and 1 then $M$ is a classical manifold-with-boundary. The $j$-corners are a $j$-fold self intersection of the $1$-corners. Locally along a $j$-fold self intersection the $j$ parts of the 1-corners intersecting there can be labelled, but this may not be possible globally along the $j$-corners. Instead there is a $\Sigma_j$-bundle over the $j$-corners which records the failure of that globalization; it is called the \textit{face-labelling bundle}.

\end{defn}

\begin{defn}
An $^nV_\infty$-manifold is a manifold-with-$j$-corners for $0\leq j\leq n$, together with a $MU$-structure on the interior of the codimension 0 stratum, a free $C_p$ action on the 1-corners preserving the $MU$-structure induced there, and such that on the $j$-corners the $C_p$ action combines to a free $C_p^j\rtimes\Sigma_j$ action on the face-labelling bundle (the semidirect product is the one associated to the permutation action of $\Sigma_j$ on $C_p^j$).
\end{defn}

\begin{defn}\label{bulletspectrainfty} (cf. \cite{bullet} Definition 1.3)
Define the spectrum $^nV_\infty$ to be the bordism theory of $^nV_\infty$-manifolds. More precisely, $^nV_\infty$ is the stable homotopy type representing the following cohomology theory on manifolds: for a manifold $X$, the group $^nV_\infty^k(X)$ is the set of \textit{cobordism} (defined shortly) classes of $\text{dim}X-k$-dimensional $^nV_\infty$-manifolds $Q$ with an $^nV_\infty$-oriented, proper map $f:Q\rightarrow X$, which means that $Q$ is presented as a submanifold of $\mathbb{R}^\infty\times X$ with a $MU$-structure on its normal bundle (on the interior of the zero-corners) and $f$ is the projection to $X$. Two such data $(Q,f)$ and $(Q',f')$ are \textit{cobordant} if there is a $\text{dim}X-k+1$-dimensional $^nV_\infty$-manifold  $R$ with an $^nV_\infty$-oriented, proper map $f:R\rightarrow X\times\mathbb{R}$ that is transverse to $X\times\{0\}$ and $X\times\{1\}$ and whose pullbacks over those submanifolds are identified with $(Q,f)$ and $(Q',f')$.  
\end{defn}

\begin{defn}\label{definingfilinfty}
Define $V_\infty $ as the colimit of $MU\simeq \ ^0V_\infty\rightarrow\ ^1V_\infty\rightarrow\ ^2V_\infty\rightarrow ...$ where the connecting maps are the maps induced by regarding an $^nV_\infty$-manifold as an $^{n+1}V_\infty$-manifold with empty $(n+1)$-corners.
\end{defn}

\begin{defn}\label{MUwreath} (cf. \cite{bullet} page 14)
Let $(MU\otimes BC_p)^{\otimes_{MU}n}_{h\Sigma_n}$ be the homotopy quotient of $(MU\otimes BC_p)^{\otimes_{MU}n}$ by the permutation action of $\Sigma_n$. Geometrically $(MU\otimes BC_p)^{\otimes_{MU}n}_{h\Sigma_n}$ is the spectrum associated to the bordism theory of $MU$-manifolds with free a $C_p^n\rtimes\Sigma_n$-action such that $C_p^n$ preserves the $MU$-structure and $\Sigma_n$ acts by the sign representation it.
\end{defn}

\begin{defn}\label{joinaction} (cf. \cite{bullet} page 11)
Let $C_p^{\star n}$ be the `unravelled' $n$-fold join of $C_p$, presented as the subset of $(C_p\times\mathbb{R})^n$ consisting of those $(z_1,t_1,...,z_n,t_n)$ such that the $t_i$ are nonnegative and sum to 1. That admits a natural action of $C_p^n\rtimes\Sigma_n$ considered as the group of $n\times n$ permutation matrices with values in $\mathbb{F}_p\simeq C_p$. Note that $C_p^{\star n}$ is the $^{n-1}V_\infty$-manifold defined by $C_p^n\times\Delta_\text{top}^{n-1}$ together with the $C_p$-action that over the $i$th face of $\Delta_\text{top}^{n-1}$ acts by translation on the $ith$ factor of $C_p^n$.
\end{defn}

\begin{rem}
In the language of \cite{bullet}, $C_p^{\star n}$ is the $^nV_\infty$-manifold obtained by `cutting along the singularity strata' of the $n$-fold join $C_p^{*n}$.
\end{rem}

\begin{lem}\label{geometricfilinfty} (cf. \cite{bullet} Proposition 2.1)
In the defining filtration of $V_\infty $ in Definition \ref{definingfilinfty}, the successive quotients are given by
$$^nV_\infty/\ ^{n-1}V_\infty\simeq \Sigma^{n}(MU\otimes BC_p)^{\otimes_{MU}n}_{h\Sigma_n}.$$
The quotient map is geometrically presented as the map that sends an $^nV_\infty$-manifold to the $MU$-manifold with free a $C_p^n\rtimes\Sigma_n$-action (cf. Definition \ref{MUwreath}) defined by its $n$-corner's face-labelling bundle.
Moreover, the attaching map
$$\Sigma^{n-1}(MU\otimes BC_p)^{\otimes_{MU}n}_{h\Sigma_n}\rightarrow\ ^{n-1}V_\infty$$
whose cofiber is $^nV_\infty$ is presented geometrically as the map that sends an $MU$-manifold $P$ with a free $C_p^n\rtimes\Sigma_n$-action\footnote{The action interacts appropriately with the $MU$-structure (cf. Definition \ref{MUwreath}).} to the $^{n-1}V_\infty$-manifold given by the $C_p^{\star n}$-bundle $P\times_{C_p^n\times\Sigma_n}C_p^{\star n}$ (cf. Definition \ref{joinaction})
\end{lem}
\begin{proof}
This is the proof given in \cite{bullet}. It suffices to show that the geometrically defined maps
$$...\rightarrow\Sigma^{n-1}(MU\otimes BC_p)^{\otimes_{MU}n}_{h\Sigma_n}\rightarrow\ ^{n-1}V_\infty\rightarrow\ ^nV_\infty\rightarrow\Sigma^{n}(MU\otimes BC_p)^{\otimes_{MU}n}_{h\Sigma_n}\rightarrow... $$
form a cofiber sequence, which is equivalent to the statement that the corresponding maps of collections of manifolds(-with-singularities) form an exact sequence up to bordism (the base manifold $X$ plays no real role except to clutter the notation). The composite of the left two maps is zero since the $^{n-1}V_\infty$-manifold $P\times_{C_p^n\rtimes\Sigma_n}C_p^{\star n}$ is null as an $^nV_\infty$-manifold (since the cone on $C_p^{\star n}$ is naturally an $^nV_\infty$-manifold). Moreover if an $^{n-1}V_\infty$-manifold $M$ is null as an $^nV_\infty$-manifold, then any such nullbordism $Q$ exhibits $M$ as bordant to a neighborhood of the $n$-corners of $Q$, which is in the image of the first map. The composite of the second two maps is clearly zero since an $^{n-1}V_\infty$-manifold has empty $n$-corners. Finally, if $M$ is an $^nV_\infty$-manifold whose $n$-corners' face-labelling bundle is null in  $(MU\otimes BC_p)^{\otimes_{MU}n}_{h\Sigma_n}$, then gluing such a nullbordism onto the $n$-corners of $M$ at one end of the product of $M$ with an interval produces an $^{n}V_\infty$-manifold that exhibits $M$ as bordant to the image of the inclusion $^{n-1}V_\infty\rightarrow\ ^nV_\infty$.
\end{proof}

\begin{defn}\label{definingfilone}
An $^nV_1$-manifold is an $^nV_\infty$-manifold $Q$ together with a decomposition\footnote{This means in particular that the labelled faces are disjoint after the 2-corners have been removed.} of the 1-corners into labelled faces $d_1Q,...,d_qQ$ (whose $r$-fold intersections are of codimension $r$) such that this labelling trivializes all the face-labelling bundles (which implies that $q\geq n'$ where $n'$ is the minimal $k$ for which $Q$ is a $^kV_\infty$-manifold). The associated bordism theory $^nV_1$ is defined analogously to Definition \ref{bulletspectrainfty}, and $V_1$ is defined as the colimit of $MU\simeq\ ^0V_1\rightarrow\ ^1V_1\rightarrow ...$ where the connecting maps are induced by viewing an $^nV_1$-manifold as an $^{n+1}V_1$-manifold with empty $(n+1)$-corners and the same labelling.
\end{defn}

\begin{rem}\label{winfty}
In \cite{bullet} (Definition 2.4) Bullett defines an $^nW$-manifold to be an $^nV_\infty$-manifold with a decomposition of the 1-faces into $n$ labelled sectors $d_1M,...,d_nM$, such that the labelling simultaneously trivializes all face-labelling bundles. He defines spectra $^nW$ and a colimit spectrum $^\infty W=\text{colim}_n\ ^nW$, where the bonding maps $^nW\rightarrow\ ^{n+1}W$ are those induced by viewing an $^nW$-manifold as an $^{n+1}W$-manifold with $d_{n+1}M=\emptyset$. Although every $^nV_1$-manifold admits the structure of an $^nW$-manifold by taking the minimal labelled decomposition, the connecting maps in the colimit are different and produce different spectra. In fact we will see that $V_1$ admits a ring structure while Bullett proves that $^\infty W$ does not (cf. \cite{bullet} page 24).
\end{rem}

\begin{rem}\label{mubcpgeom}
In the following it will be convenient to note that the spectrum $MU\otimes BC_p^{\otimes n}$ represents the bordism theory of $MU$-manifolds with a free action of $C_p^n$.
\end{rem}

\begin{lem}\label{geometricfilone} In the defining filtration of $V_1 $ in Definition \ref{definingfilone}, the successive quotients are given by
$$^nV_1/\ ^{n-1}V_1\simeq \Sigma^{n}MU\otimes BC_p^{\otimes n}.$$
The quotient map is geometrically presented as the map that sends an $^nV_1$-manifold to the $MU$-manifold with free a $C_p^n$-action (cf. Remark \ref{mubcpgeom}) defined by its $n$-corner's (trivial) face-labelling bundle.
Moreover, the attaching map
$$\Sigma^{n-1}MU\otimes BC_p^{\otimes n}\rightarrow\ ^{n-1}V_1$$
whose cofiber is $^nV_1$ is presented geometrically as the map that sends a $MU$-manifold $P$ with free a $C_p^n$-action to the $^{n-1}V_1$-manifold given by the $C_p^{\star n}$-bundle $P\times_{C_p^n}C_p^{\star n}$ (cf. Definition \ref{joinaction}).

\end{lem}
\begin{proof}
The proof is identical to that of Lemma \ref{geometricfilinfty} once it is noted that $C_p^{\star n}$ is in fact an $^{n-1}V_1$-manifold because the 1-corners of $C_p^n\times\Delta^{n-1}_\text{top}$ are naturally labelled.
\end{proof}

\begin{lem}\label{homologyofbullett}
The mod $p$ homology of $V_\infty$ and $V_1$ agree with the mod $p$ homology of their associated graded spectra (see Lemmas \ref{geometricfilinfty} and \ref{geometricfilone}). In particular, 
$(H\mathbb{F}_p)_*V_1$ is the free associative $(H\mathbb{F}_p)_*MU$-algebra on the vector space $(H\mathbb{F}_p)_*\Sigma BC_p$.
\end{lem}
\begin{proof}
It suffices to show that the geometrically defined attaching maps in Lemmas \ref{geometricfilinfty} and \ref{geometricfilone} are null after tensoring with $H\mathbb{F}_p$. Observe that the action of the diagonal $C_p$ in $C_p^{\times n}$ acts freely on $P\times_{C_p^n\times\Sigma_n}C_p^{\star n}$ (resp. $P\times_{C_p^n}C_p^{\star n}$) by its natural action on the right-hand factor, preserving the $^{n-1}V_\infty$-structure. So the quotient by that action is an $^{n-1}V_\infty$- (resp. $^{n-1}V_1$-) manifold. Recall that the transfer map $BC_p\rightarrow\mathbb{S}$ admits the geometric presentation as the map that sends a framed manifold with a $C_p$-bundle to the total space of that bundle. Hence the two attaching maps in question factor as
$$\Sigma^{n-1}(MU\otimes BC_p)^{\otimes_{MU}n}_{h\Sigma_n}\rightarrow\ ^{n-1}V_\infty\otimes BC_p\rightarrow\ ^{n-1}V_\infty$$
$$\Sigma^{n-1}MU\otimes BC_p^{\otimes n}\rightarrow\ ^{n-1}V_1\otimes BC_p\rightarrow\ ^{n-1}V_1$$
where the second maps are the transfer map (suitably tensored). Since the transfer map is null after tensoring with $H\mathbb{F}_p$ the first sentence of the lemma is proved. The second sentence of the lemma about the mod $p$ homology of $V_1$ follows immediately from the identification of the associated graded in Lemma \ref{geometricfilone}.

\end{proof}

\begin{proof}
The cartesian product of an $^nV_\infty$- and and $^mV_\infty$-manifold is an $^{n+m}V_\infty$ manifold. The cartesian product of an $^nV_1$- and and $^mV_1$-manifold $Q$ and $R$ is also naturally an $^{n+m}V_1$-manifold by the decomposition of the 1-corners of $Q\times R$ into $d_1Q\times R,...,d_qQ\times R,Q\times d_1R,...,Q\times d_rR$. Hence the natural map $V_1^*X\rightarrow V_\infty ^*X$ is a map of multiplicative cohomology theories. Lifting this structure via Brown representability gives the representing spectra a ring structure which is respected by the map between them. Since I have used Brown representability, everything is modulo phantom maps.
\end{proof}

\begin{rem}
The ring structure on $V_\infty $ is homotopy commutative (up to phantom maps), since the cartesian product of $V_\infty $-manifolds is symmetric up to cobordism. The same is \textit{not} true for $V_1$. The labelling of the $1$-corners on a cartesian product (described above) is clearly sensitive to the order of the factors in that cartesian product.
\end{rem}

\begin{rem}
Although it is probably not hard to rule out the possibility of phantom maps in the situation above, the much-refined Lemma \ref{highlystructuredbullett} makes it a moot point.
\end{rem}

\subsection{A solution to Bullett's conjecture}

\begin{defn}
Let $R_*$ denote the image in mod $p$ homology of the canonical map $V_1\rightarrow V_\infty$ induced by viewing a $V_1$-manifold as a $V_\infty$-manifold. It is a subalgebra of $(H\mathbb{F}_p)_*V_\infty$. Let $R^*$ be the degree-wise $\mathbb{F}_p$-linear dual of $R_*$. It is a quotient co-algebra of $H\mathbb{F}_p^*V_\infty$.
\end{defn}

\begin{lem}\label{rstar}
$R_*$ is isomorphic to the free commutative $(H\mathbb{F}_p)_*MU$-algebra on the $\mathbb{F}_p$-module $(H\mathbb{F}_p)_*\Sigma BC_p$.
\end{lem}
\begin{proof}
The map $i:V_1\rightarrow V_\infty$ respects the filtrations defined in Definitions \ref{definingfilinfty} and \ref{definingfilone}. On the successive quotients it is the canonical quotient $i_k:MU\otimes(BC_p)^k\rightarrow(MU\otimes (BC_p)^k)_{h\Sigma_k}$. Let $A_n^*$ denote the antisymmetric part of $H\mathbb{F}_p^*(MU\otimes(BC_p)^k)$. On mod $p$ cohomology, the image of $i_k^*$ is $H\mathbb{F}_p^*MU\otimes A_n^*$ (cf. \cite{bullet} Proposition 2.10). Combining that with Lemma \ref{homologyofbullett} finishes the proof.
\end{proof}

\begin{defn}\label{univmodporiented}
By \cite{bullet} Proposition 2.14, $R^*$ is a free module over the Steenrod algebra, so the formula $V^*X:=\text{Hom}_{\mathcal{A}_p}(R^*,H\mathbb{F}_p^*X)$ defines a cohomology theory. The \textit{universal mod $p$ oriented ring spectrum} $V$ (cf. \cite{bullet} Corollary 3.3) is the spectrum associated to that cohomology theory. By construction, $V$ is a summand of $V_\infty$, as presented by the canonical inclusion $V^*X\hookrightarrow\text{Hom}_{\mathcal{A}_p}(H\mathbb{F}_p^*V_\infty,H\mathbb{F}_p^*X)$, and $V_*\simeq L_p$ (cf. Definition \ref{modplaz}) by Corollary 3.3 of \cite{bullet}.
\end{defn}

\begin{rem}\label{universalprop}
By construction, $V^*BC_p$ carries the universal mod $p$ formal group law, and the set of homotopy ring maps $V\rightarrow R$ is in bijection with the set of mod $p$ orientations of $R$.
\end{rem}

\begin{lem}\label{freeoverSA}
The mod $p$ cohomology of $V_1$ and $V_\infty$ are free over the Steenrod algebra.
\end{lem}
\begin{proof}
It suffices to show that $V_1$ and $V_\infty$ are sums of shifts of $H\mathbb{F}_p$. By a theorem of Rourke (cf. \cite{bullet} Theorem 1.1,  \cite{rourke}) a multiplicative cohomology theory $E^*X$ is (represented by) a direct sum of shifts of $H\mathbb{F}_p$ if $E^0(\text{pt})=\mathbb{F}_p$ and $E^{>0}(\text{pt})=0$ and there is a map of cohomology theories $E^*X\rightarrow H\mathbb{F}_p^*X$ which is surjective when $X=S^{2n+1}/C_p$ for sufficiently large $n$ (note that those are finite subcomplexes of $BC_p$).

Consider the morphisms $V_1\rightarrow H\mathbb{F}_p$ and $V_\infty\rightarrow H\mathbb{F}_p$ defined by sending a $V_1$- or $V_\infty$-oriented map $Q\rightarrow X$ to the Poincare dual of the associated homology class. Write $H\mathbb{F}_p^*BC_p\simeq\mathbb{F}_p[e,x]$ with $|e|=1$ and $|x|=2$. Let $i_n:S^{2n+1}/C_p\rightarrow BC_p$ be the inclusion. 
Then the surjectivity condition in Rourke's theorem for the two maps in question reduces to the condition that for sufficiently large $n$ there are $V_1$-(and hence $V_\infty$-) manifolds $Q_{e,n}\rightarrow S^{2n+1}/C_p$ and $Q_{x,n}\rightarrow S^{2n+1}/C_p$ representing the Poincare duals of $i_n^*x$ and $i_n^*e$. 

The Poincare dual of $i_n^*x$ is represented by $Q_{x,n}:=S^{2n-1}/C_p\rightarrow S^{2n+1}/C_p$. Let $Q_{e,n}$ be the quotient of $D^{2n}$ by the action of $C_p$ on its boundary $S^{2n-1}$. Then the Poincare dual of $i_n^*e$ is represented by the map $Q_{e,n}\rightarrow S^{2n+1}/C_p$ induced by including $D^{2n}$ as a half an equator in $S^{2n+1}$. Finally, $Q_{x,n}$ is a $MU$-manifold, and $Q_{e,n}$ is a $V_1$-manifold (in fact, a $^1V_1$ manifold) so the proof is complete.
\end{proof}

\begin{rem}
There is a different proof of the above (see  \cite{modpdualsteen}), which uses formal groups and some defining universal properties of $V_1$ to show that it is an $H\mathbb{F}_p$-algebra in the stable homotopy category, and one derives the mod $p$ (dual) Steenrod algebra at the same time.
\end{rem}

\begin{thm} (Bullett's conjecture)
The spectrum $V$ admits a geometric presentation as the bordism theory of $V_\infty$-manifolds which have trivializable face-labelling bundles (cf. Definition \ref{corners}) up to cobordism.
\end{thm}
\begin{proof}
First I claim that the image of $i_X:V_1^*X\rightarrow V_\infty^*X$ coincides with $V^*X$ (cf. Definition \ref{univmodporiented}), so that $i$ factors as the projection onto a summand of $V_1$ (which must be equivalent to $V$) and the inclusion of the summand $V$ into $V_\infty$. Indeed, the mod $p$ cohomology of $V_1$ and $V_\infty$ are free over the Steenrod algebra by Lemma \ref{freeoverSA}, so $i_X$ can be written as
$$\begin{tikzcd}
V_1^*X\simeq\text{Hom}_{\mathcal{A}_p}(H\mathbb{F}_p^*V_1,H\mathbb{F}_p^*X)\arrow[r,"i^*\circ(-)"]&\text{Hom}_{\mathcal{A}_p}(H\mathbb{F}_p^*V_\infty,H\mathbb{F}_p^*X)\simeq V_\infty^*(X)\\
\end{tikzcd}$$
so its image is $\text{Hom}_{\mathcal{A}_p}(R^*,H\mathbb{F}_p^*X)$, which is the definition of $V^*X$.
So $V^*X$ is the image of $V_1^*X$ in $V_\infty^*X$, and a $V_\infty$-oriented map $[Q\rightarrow X]\in V_\infty^*X$ is in that image precisely when it admits a $V_1$-structure up to cobordism, i.e. when it is cobordant to a $Q'\rightarrow X$ whose face-labelling bundles are all trivialized.
\end{proof}

\begin{rem}
Bullett frames his conjecture on page 24 of \cite{bullet}, where he suggests that $V$ is the image of the natural map $^\infty W\rightarrow V_\infty$ (cf. Remark \ref{winfty}). That map has the same image as the map $V_1\rightarrow V_\infty$ (Compare Lemma \ref{rstar} and Proposition 2.13 of \cite{bullet}, and cf. Remark \ref{winfty}).
\end{rem}

\bibliographystyle{plain}

\bibliography{references}

\begin{thebibliography}{1}

\bibitem{bullet}
S.~Bullet.
\newblock Z/p bordism.
\newblock {\em Math. Z.}, 141:9--24, 1975.

\bibitem{modpdualsteen}
T.~Campion and K.~Luecke.
\newblock The steenrod algebras from formal groups and bordisms.
\newblock In progress.

\bibitem{quillen}
D.~Quillen.
\newblock Elementary proofs of some results of cobordism theory using steenrod
  operations.
\newblock {\em Advances in Mathematics}, 7:29--56, 1971.

\bibitem{rourke}
C.~Rourke.
\newblock Representing homology classes.
\newblock {\em Bull. London Math.}, 5:257--260, 1973.

\end{thebibliography}

\end{document}